\documentclass{amsart}
\usepackage[foot]{amsaddr}
\usepackage{graphicx}
\usepackage{ulem}
\usepackage{bm}
\usepackage{color}
\usepackage{amssymb,amsmath,amsthm,mathrsfs,mathdots,mathtools}
\usepackage{bm}
\theoremstyle{plain}
\usepackage[aboveskip=-5pt,position=bottom]{caption}
\usepackage{natbib}
\usepackage{url}
\usepackage{enumerate}
\usepackage{rotating}
\usepackage{booktabs,arydshln}
\usepackage[table,x11names]{xcolor}
\usepackage{epstopdf}
\usepackage{graphicx,verbatim,array,multicol,courier,dsfont}
\usepackage[pdftex,bookmarks,colorlinks]{hyperref}
\usepackage{color}
\usepackage[section]{placeins} 
\usepackage{soul} 
\usepackage{booktabs} 
\usepackage{paralist}
\usepackage{enumitem}
\usepackage{float}

\makeatletter
\def\adl@drawiv#1#2#3{%
        \hskip.5\tabcolsep
        \xleaders#3{#2.5\@tempdimb #1{1}#2.5\@tempdimb}%
                #2\z@ plus1fil minus1fil\relax
        \hskip.5\tabcolsep}
\newcommand{\cdashlinelr}[1]{%
  \noalign{\vskip\aboverulesep
           \global\let\@dashdrawstore\adl@draw
           \global\let\adl@draw\adl@drawiv}
  \cdashline{#1}
  \noalign{\global\let\adl@draw\@dashdrawstore
           \vskip\belowrulesep}}
\makeatother

\usepackage{accents}

\newtheorem{lemma}{Lemma}[section]

\newtheorem{theorem}[lemma]{Theorem}
\newtheorem{cor}[lemma]{Corollary}
\newtheorem{prop}[lemma]{Proposition}
\newtheorem{exam}[lemma]{\normalfont \scshape
 Example}
\newtheorem{rema}[lemma]{\normalfont \scshape Remark}
\newtheorem{defi}[lemma]{Definition}

\newcommand{\R}{\mathbb{R}}
\newcommand{\N}{\mathbb{N}}

\newcommand{\norm}[1]{\left\Vert#1\right\Vert}

\newcommand{\abs}[1]{\left\vert#1\right\vert}
\newcommand{\set}[1]{\left\{#1\right\}}

\newcommand{\bfx}{\bm{x}}
\newcommand{\bfzero}{\bm{0}}
\newcommand{\bfinfty}{\bm{\infty}}

\newcommand{\bfone}{\bm{1}}
\newcommand{\bfa}{\bm{a}}
\newcommand{\bfb}{\bm{b}}

\newcommand{\bfM}{\bm{M}}

\newcommand{\bfU}{\bm{U}}
\newcommand{\bfu}{\bm{u}}

\newcommand{\bfV}{\bm{V}}

\newcommand{\bfX}{\bm{X}}
\newcommand{\bfY}{\bm{Y}}
\newcommand{\bfy}{\bm{y}}
\newcommand{\bfZ}{\bm{Z}}

\newcommand{\bfeta}{\bm{\eta}}
\newcommand{\bfgamma}{\bm{\gamma}}

\newcommand{\calP}{\mathcal P}
\newcommand{\calI}{\mathcal I}

\newcommand{\allpart}{\mathscr P}

\newcommand{\diff}{\mathrm{d}}

\def\indic{\mathds{1}}

\usepackage[doublespacing]{setspace}

\allowdisplaybreaks[4]

\usepackage{hyperref}
\hypersetup{colorlinks,%
citecolor=blue,%
filecolor=green,%
linkcolor=red,%
urlcolor=violet,%
}

\begin{document}

\title[Strong Convergence]{Strong Convergence of Multivariate Maxima}

\author[M. Falk]{Michael Falk$^{(1)}$}
\address[1]{Institute of Mathematics, University of W\"{u}rzburg, W\"{u}rzburg, Germany}
\email{michael.falk@uni-wuerzburg.de}

\author[S.A. Padoan]{Simone A. Padoan$^{(2)}$}
\address[2]{Department of Decision Sciences,
Bocconi University of Milan,  Milano, Italy}
\email{simone.padoan@unibocconi.it}

\author[S. Rizzelli]{Stefano Rizzelli$^{(3)}$}
\address[3]{Institute of Mathematics,
\'Ecole Polytechnique F\'ed\'erale de Lausanne, Lausanne, Switzerland}
\email{stefano.rizzelli@epfl.ch}

\subjclass[2010]{Primary 60G70; secondary 62H10, 60F15}%
\keywords{Maxima; strong convergence; total variation; copula; generalized Pareto copula; $D$-norm; multivariate max-stable distribution; domain of attraction}%

\date{\today}%

\begin{abstract}
It is well known and readily seen that the maximum of $n$ independent and uniformly on $[0,1]$ distributed random variables, suitably standardised, converges in total variation distance, as $n$ increases, to the standard negative exponential distribution. We extend this result to higher dimensions by considering copulas.
We show that the strong convergence result holds for copulas that are in a differential neighbourhood of a multivariate generalized Pareto
copula.
%
%
Sklar's theorem then implies convergence in variational distance of the maximum of $n$ independent and identically distributed random vectors with arbitrary common distribution function and (under conditions on the marginals) of its appropriately normalised version. We illustrate how these convergence results can be exploited to establish the almost-sure consistency of some estimation procedures for max-stable models, using sample maxima.
\end{abstract}

\maketitle


%
\section{Introduction}\label{sec:intro}
%
%

Let $U$ be a random variable (rv), which follows the uniform distribution on $[0,1]$, i.e.,
\begin{equation}\label{eq:uniform}
P(U\le u) = \begin{cases}
0,&u<0\\
u,&u\in[0,1]\\
1,&u>1
\end{cases}\quad =: V(u).
\end{equation}
Let $U^{(1)},U^{(2)},\dots$ be independent and identically distributed (iid) copies of $U$. Then, clearly, we have for $x\le 0$ and large $n\in\N$ (natural set),
\begin{align}\label{eqn:convergence in distribution}
P\left(n\left(\max_{1\le i\le n}U^{(i)}-1\right)\le x \right)&= P\left(U_i\le 1+\frac xn,\;1\le i\le n \right)\nonumber\\
&= V^n\left(1+\frac xn\right)\nonumber\\
&= \left(1+\frac xn\right)^n\nonumber\\
&\to_{n\to\infty}G(x),
\end{align}
where
\begin{equation}\label{eqn:exponential}
G(x)=\begin{cases}
\exp(x),&x\le 0\\
1,&x>0
\end{cases}
\end{equation}
is the distribution function (df) of the standard negative exponential distribution. Thus, we have established convergence in distribution of the suitably normalised sample maixmum, i.e.
\begin{equation}\label{eqn:univariate convergence in distribution}
n\left(M^{(n)}-1\right)\to_D\eta,
\end{equation}
where $M^{(n)}:=\max_{1\le i\le n}U^{(i)}$, $n\in\N$, the arrow ``$\to_D$'' denotes convergence in distribution, and the rv $\eta$ has df $G$ in \eqref{eqn:exponential}.

Note that, with $v(x):= V'(x)=1$, if $x\in[0,1]$, and zero elsewhere, we have
\begin{align*}
v_n(x)&:=\frac{\partial}{\partial x}\left(V^n\left(1+\frac xn\right)\right)= V^{n-1}\left(1+\frac xn \right) v\left(1+\frac xn \right)\\
&\to_{n\to\infty} g(x):= G'(x) = \begin{cases}
\exp(x),&x\le 0\\
0,&x>0
\end{cases},
\end{align*}
i.e., we have pointwise convergence of the sequence of densities of normalised maximum $n\left(M^{(n)}-1\right)$, $n\in\N$, to that of $\eta$. Scheff\'{e}'s lemma, see, e.g. \citet[Lemma 3.3.3]{reiss89} now implies convergence in total variation:
\begin{equation}\label{eqn:strong convergence, the univariate case}
\sup_{A\in\mathbb B}\abs{P\left(n\left(M^{(n)}-1\right)\in A \right)-P(\eta\in A)}\to_{n\to\infty}0,
\end{equation}
where $\mathbb B$ denotes the Borel-$\sigma$-field in $\R$.

Let now $X$ be a rv with \textit{arbitrary} df $F$ and $F^{-1}(q):=\set{t\in\R:\, F(t)\ge q}$ with $q\in (0,1)$ be the usual \textit{quantile function} or \textit{generalized inverse} of $F$.
Then, we can assume the representation
\[
X=F^{-1}(U).
\]
%
Let $X^{(1)},X^{(2)},\dots$ be independent copies of $X$. Again we can consider the representation
\[
X^{(i)}=F^{-1}\left(U^{(i)}\right),\qquad i=1,2,\dots
\]
The fact that each quantile function is
a nondecreasing function yields
\begin{align*}
\max_{1\le i\le n}X^{(i)}&=\max_{1\le i\le n}F^{-1}\left(U^{(i)}\right) = F^{-1}\left(\max_{1\le i\le n}U^{(i)} \right)\\
&= F^{-1}\left(1+\frac 1n\left(n\left(\max_{1\le i\le n}U^{(i)}-1 \right) \right) \right).
\end{align*}
The strong convergence in equation \eqref{eqn:strong convergence, the univariate case} now implies the following convergence in total variation:
\begin{equation}\label{eqn:strong_convergence_the general univariate case}
\sup_{A\in\mathbb B}\abs{P\left(\max_{1\le i\le n}X^{(i)}\in A\right)- P\left(F^{-1}\left(1+\frac 1n \eta\right)\in A\right)}\to_{n\to\infty}0.
\end{equation}

Finally, assume that $F$ is a continuous df with density $f=F'$. We denote the right endpoint of $F$ by $x_0:=\sup\{x \in \R: \, F(x)<1\}$. Assume also that $F\in\mathcal{D}(G_\gamma^{*})$, i.e. $F$ belongs to the domain of attraction of a generalised extreme-value df $G_\gamma^{*}$, e.g. \citet[][p. 21]{falkpadoan2018}.
This means, for $n\in\N$, there are norming constants $a_n>0$ and $b_n\in\R$ such that
\begin{equation}\label{eq:gev_family}
F^n(a_nx+b_n)\to_{n\to\infty} \exp\left(-(1+\gamma x)_+^{-1/\gamma}\right)=:G^*_\gamma(x),
\end{equation}
for all $x\in\R$, where $(x)_+=\max(0,x)$ and $\gamma\in\R$ is the so-called {\it tail index}.
Such a coefficient describes the heaviness of the upper tail of the probability density function corresponding to $G^*_\gamma$, see \citet[][for details]{falkpadoan2018}.
Furthermore, in this general case, we also have the pointwise convergence at the density level, i.e.
\begin{align}\label{eq: conv_dens}
f^{(n)}(x):=\frac{\partial}{\partial x}F^n(a_nx +b_n)\to_{n \to \infty}  \frac{\partial}{\partial x}G^*_\gamma(x) =:g^*_\gamma(x)
\end{align}
for all $x\in\R$, if and only if
\begin{align}
\label{eq: restr_frec}
&\lim_{x \to \infty} \frac{x f(x)}{1-F(x)}=1/\gamma, \quad &\text{if }\gamma>0\\
\label{eq: restr_weib}
&\lim_{x \uparrow x_0}\frac{(x_0-x)f(x)}{1-F(x)}=-1/\gamma, \quad &\text{if }\gamma<0\\
\label{eq: restr_gumb}
&\lim_{x \uparrow x_0}\frac{f(x)}{(1-F(x))^2}\int_0^{x_0}1-F(t)\diff t=1, \quad &\text{if } \gamma=0,
\end{align}
see e.g. Proposition 2.5 in \cite{resn08}.
In particular, if \eqref{eq: conv_dens} holds true, Scheff\'{e}'s lemma entails that
\begin{equation}\label{eqn:strong_convergence_the_gev_case}
\sup_{A\in\mathbb B}\abs{P\left(a_n^{-1}\left(\max_{1\le i\le n}X^{(i)}-b_n\right)\in A\right)- P\left(Y\in A\right)}\to_{n\to\infty}0,
\end{equation}
{where $Y$ is a rv with distribution $G_\gamma^{*}$ and $X^{(i)}$, $i=1,\ldots,n$ are independent copies of a rv $X$ with distribution $F$, with $F\in\mathcal{D}(G_\gamma^{*})$.

In this paper we extend the results in \eqref{eqn:strong convergence, the univariate case}, \eqref{eqn:strong_convergence_the general univariate case} and \eqref{eqn:strong_convergence_the_gev_case} to higher dimensions.
First, in Section \ref{sec:strong results for copulas}, we consider copulas.
In Theorem \ref{theo:extension to copulas in a neighborhood}, we demonstrate that the strong convergence result holds for copulas that are in a differential neighbourhood of a multivariate generalized Pareto copula \citep{falkpadoan2018,falk2019}.
As a result of this, we also establish strong convergence of the copula of the maximum of $n$ iid random vectors with arbitrary common df to the limiting extreme-value copula (Corollary \ref{cor: max_copula}).
%
%
Sklar's theorem is then used in Section \ref{sec:general case}  to derive convergence in variational distance of the maximum of $n$ iid random vectors with arbitrary common df and, under restrictions \eqref{eq: restr_frec}-\eqref{eq: restr_gumb} on the margins, of its normalised versions. These results address some still open problems in the literature on multivariate extremes.

Strong convergence for extremal order statistics of univariate iid rv has been well investigated; see, e.g. Section 5.1 in \citet{reiss89} and the literature cited therein. Strong convergence holds in particular under suitable von Mises type conditions on the underlying df, see \eqref{eq: restr_frec}-\eqref{eqn:strong_convergence_the_gev_case} for the univariate normalised maximum. Much less is known in the multivariate setup. In this case, a possible approach is to investigate a point process of exceedances over high thresholds and establish its convergence to  a Poisson process.
This is done under suitable assumptions on variational convergence for truncated point measures, see e.g. Theorem 7.1.4 in \citet{fahure10}. It is proven in \citet{kaufr93} that strong convergence of such multivariate point processes holds if, and only if, strong convergence of multivariate maxima occurs.
Differently from that, we provide simple conditions (namely \eqref{eqn:crucial expansion of copula} and \eqref{eq:limit_under_der}) under which strong convergence of multivariate maxima and its normalised version actually holds. Furthermore, our strong convergence results for sample  maxima are valid for maxima with arbitrary dimensions, unlike those in \cite{DEHAAN1997195}, which are tailored to the two-dimensional case.
Section \ref{sec:applications} concludes the paper, by illustrating how effective our variational convergence results are for statistical purposes. In particular, when the interest is on inferential procedures for sample maxima whose df is in a neighborhood of some multivariate max-stable model, we show that, e.g., our results can be used to establish almost-sure consistency for the empirical copula estimator of the extreme-value copula. Similar results can also be achieved within the Bayesian inferential approach.

%



%
\section{Strong Results for Copulas}\label{sec:strong results for copulas}
%
%

Suppose that the random vector (rv) $\bfU=(U_1,\dots,U_d)$ follows a \textit{copula}, say $C$, on $\R^d$, i.e., each component $U_j$ has df $V_j$ given in formula \eqref{eq:uniform}. Let $\bfU^{(1)},\bfU^{(2)},\dots$ be independent copies of $\bfU$ and put for $n\in\N$
\begin{equation}\label{eqn:componentwise_max}
\bfM^{(n)}:= \left(M_1^{(n)},\dots,M_d^{(n)}\right):= \left(\max_{1\le i\le n}U_1^{(i)},\dots,\max_{1\le i\le n}U_d^{(i)}\right).
\end{equation}

In the sequel the operations involving vectors are meant componentwise, furthermore,
we set $\bfzero=(0,\ldots,0)$, $\bfone=(1\ldots,1)$ and $\bfinfty=(\infty,\ldots,\infty)$.
Finally, hereafter, we denote the copula of the random vector in \eqref{eqn:componentwise_max} by $C^n(\bfu)$, $\bfu \in [0,1]^d$.

Suppose that a convergence result analogous to \eqref{eqn:convergence in distribution} holds for the random \textit{vector} $\bfM^{(n)}$ of componentwise maxima, i.e., suppose there exists a nondegenerate df $G$ on $\R^d$ such that for $\bfx=(x_1,\dots,x_d)\le\bfzero\in\R^d$
\begin{align}\label{eqn:multivariate convergence in distribution}
P\left(n\left(\bfM^{(n)}-\bfone\right)\le \bfx\right) &= P\left(n\left(M_1^{(n)}-1\right)\le x_1, \dots, n\left(M_d^{(n)}-1\right)\le x_d\right)\nonumber\\
&\to_{n\to\infty} G(\bfx).
\end{align}
Then, $G$ is necessarily a {\it multivariate max-stable} or \textit{multivariate extreme-value df}, with {\it extreme-value copula} $C_G$ and standard negative exponential margins $G_j$, $j=1,\ldots,d$, see \eqref{eqn:exponential}.
In the sequel we refer to the df $G$ in \eqref{eqn:multivariate convergence in distribution} as
{\it standard} multivariate max-stable df.
Precisely, the form of $G$ is
$$
G(\bfx)=C_G(G_1(x_1),\ldots,G_d(x_d)),
$$
where the copula $C_G$ can be expressed in terms of $\norm\cdot_D$, a {\it$D$-norm} on $\R^d$,   via
\begin{equation}\label{eqn:extreme_value_copula}
C_G(\bfu)=\exp\left(-\norm{\log u_1,\ldots,\log u_d}_D\right), \quad \bfu\in[0,1]^d,
\end{equation}
while the margins $G_j$, $j=1,\ldots,d$, are as in \eqref{eqn:exponential}.
Therefore, the distribution in \eqref{eqn:multivariate convergence in distribution} has the representation
\begin{equation}\label{eq: representation}
G(\bfx)=\exp\left(-\norm{\bfx}_D\right),\qquad \bfx\le\bfzero\in\R^d.
\end{equation}

The convergence result in \eqref{eqn:multivariate convergence in distribution} implies that $C^{(n)}(\bfu):=C^n(\bfu^{1/n}) \to_{n\to\infty}C_G(\bfu)$, for all $\bfu\in[0,1]^d$, see e.g. \citet[][Corollary 3.1.12]{falk2019}. For brevity, with a little abuse of notation we also denote this latter fact by $C\in\mathcal{D}(C_G)$.
By Theorem 2.3.3 in \citet{falk2019}, there exists a rv
$\bfZ=(Z_1,\dots,Z_d)$ with $Z_i\ge 0$, $E(Z_i)=1$ , $1\le i\le d$, such that
\[
\norm{\bfx}_D= E\left(\max\left(\abs{x_i}Z_i\right) \right),\qquad \bfx\in\R^d.
\]
Examples of $D$-norms are the sup-norm $\norm{\bfx}_\infty=\max_{1\le i\le d}\abs{x_i}$, or the complete family of logistic norms $\norm{\bfx}_p=\left(\sum_{i=1}^d\abs{x_i}^p\right)^{1/p}$, $p\ge 1$. For a recent account on multivariate extreme-value theory and $D$-norms we refer to \citet{falk2019}. In particular, Proposition 3.1.5 in \citet{falk2019} implies that the convergence result in \eqref{eqn:multivariate convergence in distribution} is also equivalent to the expansion
\begin{equation}\label{eqn:crucial expansion of copula}
C(\bfu) = 1- \norm{\bfone-\bfu}_D + o(\norm{\bfone-\bfu})
\end{equation}
as $\bfu\to\bfone\in\R^d$, uniformly for $\bfu\in[0,1]^d$.

In a first step we drop the term $o(\norm{\bfone-\bfu})$ in expansion \eqref{eqn:crucial expansion of copula} and require that there exists $\bfu_0\in (0,1)^d$, such that
\begin{equation}\label{eqn:definition of generalized Pareto copula}
C(\bfu)=1-\norm{\bfone-\bfu}_D, \qquad \bfu\in[\bfu_0,\bfone]\subset\R^d.
\end{equation}
A copula, which satisfies the above expansion is a {\it generalized Pareto copula} (GPC). The significance of GPCs for multivariate extreme-value theory is explained in \citet{falkpadoan2018} and in \citet[Section 3.1]{falk2019}.

Note that
\[
C(\bfu)= \max\left(0,1-\norm{\bfone-\bfu}_D\right),\qquad \bfu\in[0,1]^d,
\]
defines a multivariate df only in dimension $d = 2$, see, e.g., \citet[Examples 2.1, 2.2]{mcnn09}. But one can find for arbitrary dimension $d \ge 2$ a rv,
whose df satisfies equation \eqref{eqn:definition of generalized Pareto copula}, see e.g. \citet[equation 2.15]{falk2019}. For this reason,
we require the condition in \eqref{eqn:definition of generalized Pareto copula} only on some upper interval $[\bfu_0,\bfone]\subset\R^d$.

The df of $n \left(\bfM^{(n)}-\bfone\right)$ is, for $\bfx<\bfzero\in\R^d$ and $n$ large so that $\bfone + \bfx/n\ge\bfu_0$,
\[
P\left(n\left(\bfM^{(n)}-\bfone\right)\le\bfx\right) = \left(1-\frac 1n\norm{\bfx}_D\right)^n =: F^{(n)}(\bfx).
\]
Suppose that the norm $\norm\cdot_D$ has partial derivatives of order $d$.  Then the df $F^{(n)}(\bfx)$ has for $\bfone +\bfx/n\ge{\bfu_0}$ the density
\begin{equation}\label{eq:density_doa}
f^{(n)}(\bfx):=\frac{\partial^d}{\partial x_1\dots\partial x_d}F^{(n)}(\bfx)= \frac{\partial^d}{\partial x_1\dots\partial x_d} \left(1-\frac 1n\norm{\bfx}_D\right)^n.
\end{equation}
%
As for the standard multivariate max-stable df $G$ in \eqref{eq: representation}, its density exists and is given by
\begin{equation}\label{eq:max_stab_den}
g(\bfx):=\frac{\partial^d}{\partial x_1\dots\partial x_d}G(\bfx)= \frac{\partial^d}{\partial x_1\dots\partial x_d} \exp\left(-\norm{\bfx}_D\right),\qquad \bfx\le\bfzero\in\R^d.
\end{equation}

We are now ready to state our first multivariate extension of the convergence in total variation in equation \eqref{eqn:strong convergence, the univariate case}. For brevity, we occasionally denote with the same letter a Borel measure and its distribution function.

\begin{theorem}\label{theo:main theorem for GPD}
Suppose the rv $\bfU$ follows a GPC $C$ with corresponding $D$-norm $\norm\cdot_D$, which has partial derivatives of order $d\ge 2$. Then
\[
\sup_{A\in\mathbb B^d}\abs{P\left(n\left(\bfM^{(n)}-\bfone\right)\in A\right)-G(A)}\to_{n\to\infty} 0,
\]
where $\mathbb B^d$ denotes the Borel-$\sigma$-field in $\R^d$.
\end{theorem}

\begin{rema}{\upshape Note that we can write a GPC
\[
C(\bfu)=1-\norm{\bfone-\bfu}_p= 1- \left(\sum_{i=1}^d(1-u_i)^p\right)^{1/p},\qquad \bfu\in[\bfu_0,\bfone]\subset\R^d,
\]
where the $D$-norm $\norm\cdot_D$ is a logistic norm $\norm\cdot_p$, $p\ge 1$, as an {\it Archimedean} copula
\[
C(\bfu)=\varphi^{-1}\left(\sum_{i=1}^d\varphi(u_i)\right), \qquad \bfu\in[\bfu_0,\bfone]\subset\R^d.
\]
The {\it generator function} $\varphi:(0,1]\to[0,\infty)$ is in general strictly decreasing and convex, with $\varphi(1)=0$ (see, e.g. \citealt{mcnn09}). Just set here $\varphi(u):=(1-u)^p$, $u\in[0,1]$. Note that we require the Archimedean structure of $C$ only in its upper tail; this allows the incorporation of $\varphi(u)=(1-u)^p$ as a generator function in arbitrary dimension $d\ge 2$, not only for $d=2$. The partial differentiability condition on the $D$-norm in Theorem \ref{theo:main theorem for GPD} now reduces to the existence of the derivative of order $d$ of $\varphi(u)$ in a left neighbourhood of $1$.}
\end{rema}

For the proof of Theorem \ref{theo:main theorem for GPD} we establish  the following auxiliary result.
\begin{lemma}\label{lem:point_con_den}
Choose $\varepsilon\in(0,1)$ and $\bfx_\varepsilon<\bfzero\in\R^d$ with $G([\bfx_\varepsilon,\bfzero])\ge 1-\varepsilon$.
Then we have for $\bfx\in[\bfx_\varepsilon,\bfzero]$
\begin{equation}\label{eqn:convergence_densities}
f^{(n)}(\bfx)\to_{n\to\infty} g(\bfx).
\end{equation}
\end{lemma}
\begin{proof}
$G(\bfx)$ can be seen as the function composition $(\ell \circ \phi)(\bfx)$, where we set $\ell(y)=\exp(y)$ and $\phi(\bfx)=-\norm{\bfx}_D$. Then, by {\it Fa\'{a} di Bruno's formula}, the density in \eqref{eq:max_stab_den} is equal to
\begin{equation}\label{eq:faadibruno_g}
g(\bfx)=\frac{\partial^d}{\partial x_1\dots\partial x_d}\exp(\phi(\bfx))=
G(\bfx)\sum_{\calP\in\allpart}\prod_{B\in\calP}\frac{\partial^{|B|}\phi(\bfx)}{\partial^{B}\bfx},
\end{equation}
where $\allpart$ is the set of all partitions of $\{1,\ldots,d\}$ and the product is over all blocks
$B$ of a partition $\calP\in\allpart$. In particular, $B=(i_1,\ldots,i_k)$ with each $i_j\in \{1,\ldots,d\}$, and
the cardinality of each block is denoted by $|B|=k$. Finally, for a function $h:\R^d\rightarrow \R$ we
define $\partial^{|B|}h/\partial^{B}\bfx:=\partial^k h / \partial x_{i_1},\ldots, \partial x_{i_k}$.

Similarly, $F^{(n)}(\bfx)$ can be seen as the function composition $(\ell\circ \phi_n)(\bfx)$, where we
set $\phi_n(\bfx):=-n\log(1/(1-n^{-1}\norm{\bfx}_D))$. Then, $F^{(n)}(\bfx)=\exp(\phi_n(\bfx))$
and, once again by the Fa\'{a} di Bruno's formula, the density in \eqref{eq:density_doa} is equal to
\begin{align*}
f^{(n)}(\bfx)=\frac{\partial^d}{\partial x_1\dots\partial x_d}\exp(\phi_n(\bfx))=
F^{(n)}(\bfx)\sum_{\calP\in\allpart}\prod_{B\in\calP}\frac{\partial^{|B|}\phi_n(\bfx)}{\partial^{B}\bfx}.
\end{align*}
Clearly, $F^{(n)}(\bfx)\to_{n\to\infty}G(\bfx)$ for all $\bfx\in[\bfx_\varepsilon,\bfzero]$.
Next, $\phi_n(\bfx)$ can be seen as the function composition $(\sigma_n\circ \phi)(\bfx)$, where
we set $\sigma_n(y)=-n\log(1/(1+n^{-1}y))$.
Thus, again by the Fa\'{a} di Bruno's formula we have that for each block $B$
%
%
\begin{align*}
\frac{\partial^{|B|}\phi_n(\bfx)}{\partial^{B}\bfx}&= \sum_{\calP_B\in\allpart_B}\frac{\partial^{|\calP_B|}\sigma_n(y)}{\partial y^{|\calP_B|}}\bigg|_{y=\phi(\bfx)}\prod_{b\in\calP_B}\frac{\partial^{|b|}\phi(\bfx)}{\partial^{b}\bfx},
\end{align*}
where $\allpart_B$ is the set of all partitions of $B=(i_1,\ldots,i_k)$ and the product is over all blocks
$b$ of partition $\calP_B\in\allpart_B$. It is not difficult to check that
$$
\frac{\partial^{|\calP_B|}\sigma_n(y)}{\partial y^{|\calP_B|}}=(-1)^{1+|\calP_B|}\,(|\calP_B|-1)!\left(1+y/n\right)^{-|\calP_B|}\,n^{-|\calP_B|+1}.
$$
Then,
\begin{equation*}
\frac{\partial^{|\calP_B|}\sigma_n(y)}{\partial y^{|\calP_B|}}\to_{n\to\infty}
\begin{cases}
1, & \text{if } |\calP_B|=1,\\
0, & \text{if } |\calP_B|>1.
\end{cases}
\end{equation*}
Notice that $|\calP_B|=1$ when $\calP_B=B$ and in this case $b=B$.
Consequently, for all $\bfx\in[\bfx_\varepsilon,\bfzero]$, we have
\[
\frac{\partial^{|B|}\phi_n(\bfx)}{\partial^{B}\bfx}\to_{n\to\infty} \frac{\partial^{|B|}\phi(\bfx)}{\partial^{B}\bfx}.
\]
Therefore, the pointwise convergence in \eqref{eqn:convergence_densities} follows.
\end{proof}

\begin{proof}[Proof of  Theorem \ref{theo:main theorem for GPD}]
It is sufficient to consider $A\subset\mathbb B^d\cap(-\infty,0]^d$, where $\mathbb B^d$ denotes the Borel-$\sigma$-field in $\R^d$.
Moreover, choose $\varepsilon>0$ and $\bfx_\varepsilon<\bfzero\in\R^d$ with $G([\bfx_\varepsilon,\bfzero])\ge 1-\varepsilon$.

We already know that
\[
\sup_{\bfx\le\bfzero}\abs{P\left(n\left(\bfM^{(n)}-\bfone\right)\le\bfx\right)-G(\bfx)}\to_{n\to\infty}0,
\]
which implies
\begin{equation}\label{eqn:convergence of maxima}
\abs{P\left(n\left(\bfM^{(n)}-\bfone\right)\in[\bfx_\varepsilon,\bfzero]\right)-G([\bfx_\varepsilon,\bfzero])}\to_{n\to\infty}0
\end{equation}
and, thus,
\[
\limsup_{n\to\infty}P\left(n\left(\bfM^{(n)}-\bfone\right)\in[\bfx_\varepsilon,\bfzero]^\complement\right)\le \varepsilon
\]
or
\begin{align*}
&\limsup_{n\to\infty} \sup_{A\in\mathbb B^d\cap[\bfx_\varepsilon,\bfzero]^\complement}\abs{P\left(n\left(\bfM^{(n)}-\bfone\right)\in A\right)-G(A)}\\
&\le \limsup_{n\to\infty}P\left(n\left(\bfM^{(n)}-\bfone\right)\in [\bfx_\varepsilon,\bfzero]^\complement\right) + G\left([\bfx_\varepsilon,\bfzero]^\complement \right) \le 2\varepsilon.
\end{align*}
As $\varepsilon>0$ was arbitrary, it is therefore sufficient to establish
\[
\sup_{A\in\mathbb B^d\cap[\bfx_\varepsilon,\bfzero]}\abs{P\left(n\left(\bfM^{(n)}-\bfone\right)\in A\right) - G(A)}\to_{n\to\infty} 0.
\]

Now, from equation \eqref{eqn:convergence of maxima} we know that
\[
\int_{[\bfx_\varepsilon,\bfzero]} f^{(n)}(\bfx)\,dx\to_{n\to\infty} \int_{[\bfx_\varepsilon,\bfzero]}g(\bfx)\,d\bfx.
\]
Together with \eqref{eqn:convergence_densities}, we can apply Scheff\'{e}'s lemma and obtain
\[
\int_{[\bfx_\varepsilon,\bfzero]}\abs{f^{(n)}(\bfx)-g(\bfx)}\,d\bfx\to_{n\to\infty} 0.
\]
The bound
\[
\sup_{A\in\mathbb B^d\cap[\bfx_\varepsilon,\bfzero]}\abs{P\left(n\left(\bfM^{(n)}-\bfone\right)\in A\right) - G(A)}\le \int_{[\bfx_\varepsilon,\bfzero]}\abs{f^{(n)}(\bfx)-g(\bfx)}\,d\bfx
\]
now implies the assertion of  Theorem \ref{theo:main theorem for GPD}.
\end{proof}

Next we extend Theorem \ref{theo:main theorem for GPD} to a copula $C$, which is in a {\it differentiable neighborhood} of a GPC, defined next.
Suppose that $C$ satisfies expansion \eqref{eqn:crucial expansion of copula}, where the $D$-norm $\norm\cdot_D$ on $\R^d$ has partial derivatives of order $d$. Assume also that $C$ is such that for each nonempty block of indices $B=(i_1,\dots,i_k)$ of $\set{1,\dots,d}$,
\begin{equation}\label{eq:limit_under_der}
\frac{\partial^k}{\partial x_{i_1},\dots,\partial x_{i_k}}n
%
\left(C\left(\bfone+\frac{\bfx}n\right)-1\right)
\to_{n\to\infty}
\frac{\partial^k}{\partial x_{i_1},\dots,\partial x_{i_k}}\phi(\bfx),
\end{equation}
holds true for all $\bfx<\bfzero\in\R^d$, where $\phi(\bfx)=-\norm{\bfx}_D$.
\begin{theorem}\label{theo:extension to copulas in a neighborhood}
Suppose the copula $C$ satisfies conditions \eqref{eqn:crucial expansion of copula} and \eqref{eq:limit_under_der}. Then we obtain
\[
\sup_{A\in\mathbb B^d}\abs{P\left(n\left(\bfM^{(n)}-\bfone\right)\in A \right)-  G(A)}\to_{n\to\infty}0,
\]
where $G$ is the standard max-stable distribution with corresponding $D$-norm $\norm\cdot_D$, i.e., it has df $G(\bfx)=\exp(-\norm{\bfx}_D)$, $\bfx\le\bfzero\in\R^d$.
\end{theorem}
\begin{proof}
The proof of Theorem \ref{theo:extension to copulas in a neighborhood} is similar to that
of Theorem \ref{theo:main theorem for GPD}, but this time we resort to a variant of Lemma \ref{lem:point_con_den} as follows.
Note that for $n\in\N$,
\[
P\left(n\left(\bfM^{(n)}-\bfone\right)\le \bfx\right) = C^n\left(\bfone+\frac{\bfx}n\right),\qquad \bfx\le\bfzero\in\R^d.
\]
Moreover, $C^n\left(\bfone+\bfx/n\right)$ is the function composition $(\ell\circ \phi_n)(\bfx)$, where
we now set $\phi_n(\bfx):=n\log\left(C\left(\bfone+\bfx/n\right)\right)$. Furthermore, $\phi_n(\bfx)$
is the composition function $(\sigma_n\circ v_n)(\bfx)$, where
we set $v_n(\bfx):=n (C(\bfone + \bfx/n)-1)$ and $\sigma_n$ is as in the proof of Lemma \ref{lem:point_con_den}.
Then, in the Fa\'{a} di Bruno's formula we have that for each block $B$,
\begin{align*}
\frac{\partial^{|B|}\phi_n(\bfx)}{\partial^{B}\bfx}&= \sum_{\calP_B\in\allpart_B}\frac{\partial^{|\calP_B|}\sigma_n(y)}{\partial y^{|\calP_B|}}\bigg|_{y=v_n(\bfx)}\prod_{b\in\calP_B}\frac{\partial^{|b|}v_n(\bfx)}{\partial^{b}\bfx}.
\end{align*}
%
%
%
By assumptions \eqref{eqn:crucial expansion of copula} and \eqref{eq:limit_under_der} we
obtain that, for each block $b$ of a partition $\calP_B\in\allpart_B$,
$$
\frac{\partial^{|b|}v_n(\bfx)}{\partial^{b}\bfx} \to_{n\to\infty} \frac{\partial^{|b|}\phi(\bfx)}{\partial^{b}\bfx}, \quad \bfx<\bfzero\in\R^d.
$$
Therefore, as in Lemma \ref{lem:point_con_den}, we obtain
\begin{equation}\label{eq:conv_copulas}
\frac{\partial^{|B|}\phi_n(\bfx)}{\partial^{B}\bfx}\to_{n\to\infty} \frac{\partial^{|B|}\phi(\bfx)}{\partial^{B}\bfx}, \quad \bfx<\bfzero\in\R^d.
\end{equation}
and the result follows.
\end{proof}

\begin{exam}\label{exam:Gumbel_copula}
{\upshape
Consider, the {\it Gumbel-Hougaard family} $\set{C_p:\,p\ge 1}$ of Archimedean copulas, with generator function $\varphi_p(u):=(-\log(u))^p$, $p\ge 1$. This is an extreme-value family of copulas. In particular, we have
\begin{align*}
C_p(\bfu)=\exp\left(-\left(\sum_{i=1}^d(-\log(u_i))^p\right)^{1/p}\right)
= 1- \norm{\bfone-\bfu}_p + o(\norm{1-\bfu}),
\end{align*}
as $\bfu\in(0,1]^d$ converges to $\bfone\in\R^d$, i.e., condition \eqref{eqn:crucial expansion of copula} is satisfied, where the $D$-norm is the logistic norm $\norm\cdot_p$ and the limiting distribution is
$G(\bfx)=\exp(-\norm\bfx_p)$.
The copula $C_p$ also satisfies conditions \eqref{eq:limit_under_der}. To prove it, we express $C_p\left(\bfone+\bfx/n\right)$ as the function composition $(\ell\circ \varphi_n)(\bfx)$, with $\ell$ as in the proof of Lemma \ref{lem:point_con_den} and $\varphi_n(\bfx):=\log\left(C_p\left(\bfone+\bfx/n\right)\right)$. Observe that
%
$$
n
\varphi_n(\bfx)=-n\norm{\log\left(1+\frac{\bfx}n\right)}_p =: -nt(s_n(\bfx)),
$$
where $t(\cdot)=\norm\cdot_p$, $s_n(\bfx)=(s_n(x_1),\ldots,s_n(x_d))$, and
$s_n(\cdot)=\log(1+\cdot /n)$.
Hence, applying the Fa\'{a} di Bruno's formula to the partial derivatives of $n (\ell\circ \varphi_n(x)-1)$ and noting that, on one hand, $C_p\left(\bfone+\bfx/n\right)\to_{n\to\infty} 1$, on the other hand,
\begin{align*}
&\frac{\partial^k}{\partial x_{i_1},\dots,\partial x_{i_k}} n\varphi_n(\bfx)
\\
&=
-n\frac{\partial^k}{\partial y_{i_1},\dots,\partial y_{i_k}} t(\bfy)\big|_{\bfy=s_n(\bfx)}
\frac{\partial s_n(x_{i_1})}{\partial x_{i_1}}\cdot \cdots \cdot \frac{\partial s_n(x_{i_k})}{\partial x_{i_k}}\\
&\simeq-n \prod_{j=1}^{k-1}(1-jp)\norm\bfx_p^{1-kp} n^{kp-1} \prod_{j=1}^k \frac{|x_{i_j}|^p}{x_{i_j}} n^{-k(p-1)}
\prod_{j=1}^k\left(1+\frac{x_{i_j}}{n}\right)^{-1}n^{-k}\\
&\to_{n\to\infty} -\frac{\partial^k}{\partial x_{i_1},\dots,\partial x_{i_k}}\norm\bfx_p,
\end{align*}
the desired result obtains. In particular, notice that we pass from the first to second line of the above display by computing partial derivatives, then from the second to the third one by exploiting the asymptotic equivalence $\log(1+y)\simeq y$, for $y \to 0$.
}
\end{exam}
\begin{exam}\label{exam:DoA_copula}
{\upshape
Consider the copula
\begin{equation}\label{eq:exponential_copula}
C(\bfu)=1-d+\sum_{i=1}^d u_i+\sum_{2\leq i\leq d}\left((-1)^i\sum_{\substack{B\subseteq\{1,\ldots,d\}\\ |B|=i}}\left(\sum_{j\in B} \frac{1}{1-u_j}-d+1\right)^{-1}\right).
\end{equation}
This provides the $d$-dimensional version (with $d\geq2$) of the $2$-dimensional copula associated to the df discussed in \cite[][Example 5.14]{resn08}. It can be checked that $C\in\mathcal{D}(C_G)$, where $C_G$
is, for all $\bfu\in[0,1]^d$,
the extreme-value copula
\begin{equation}\label{eq:limit_exp_copula}
C_G(\bfu)=\exp\left(\sum_{i=1}^d \log u_i+\sum_{2\leq i\leq d}\left((-1)^{i+1}\sum_{\substack{B\subseteq\{1,\ldots,d\}\\ |B|=i}}\left(\sum_{j\in B} \frac{1}{\log u_j}-d+1\right)^{-1}\right)\right).
\end{equation}
Then, by \citet[][Proposition 3.1.5 and Corollary 3.1.12]{falk2019} the copula in \eqref{eq:exponential_copula} satisfies condition \eqref{eqn:crucial expansion of copula},
with $D$-norm
$$
\norm{\bfx}_D=\sum_{i=1}^d|x_i|+\sum_{2\leq i\leq d}\left((-1)^{i+1}\sum_{\substack{B\subseteq\{1,\ldots,d\}\\ |B|=i}}\left(\sum_{j\in B} \frac{1}{|x_j|}\right)^{-1}\right).
$$
The copula in \eqref{eq:exponential_copula} also complies conditions in \eqref{eq:limit_under_der}. Indeed, for $2 \leq k \leq d$,
$$
\frac{\partial^k}{\partial x_{i_1},\dots,\partial x_{i_k}}\norm{\bfx}_D=
\sum_{k\leq j\leq d}\left(
(-1)^{j+1}k!\sum_{\substack{\calI \subseteq B \subseteq \{1, \ldots,d\}\\ |B|=j}}\left(\sum_{l\in B}\frac{1}{|x_l|}\right)^{-(k+1)}\prod_{v=1}^k\frac{1}{x_{i_v}^2}\frac{|x_{i_v}|}{x_{i_v}}
\right),
$$
where $\calI=\{i_1,\ldots,i_k\}$. When $k=1$, $(\partial/\partial x_{i_k})\norm{\bfx}_D $ is given by the right-hand side of the above expression plus the term $|x_{i_k}|/x_{i_k}$. Furthermore, for $2 \leq k \leq d$,
\begin{align*}
&\frac{\partial^k}{\partial x_{i_1},\dots,\partial x_{i_k}}C(\bfone+\bfx/n)\\
&=\frac{1}{n}\left(\sum_{k\leq j\leq d}\left(
(-1)^{j+1}k!\sum_{\substack{\calI \subseteq B \subseteq \{1, \ldots,d\}\\ |B|=j}}\left(\sum_{l\in B}\frac{1}{x_l}+\frac{d-1}{n}\right)^{-(k+1)}\prod_{v=1}^k\frac{1}{x_{i_v}^2}
\right)\right).
\end{align*}
When $k=1$, $n(\partial/\partial x_{i_k}) C(\bfone+\bfx/n)$ is given by the right-hand side of the above expression plus $1$. Therefore, for $k=1, \ldots, d$, we have that
$$
n\frac{\partial^k}{\partial x_{i_1},\dots,\partial x_{i_k}}C(\bfone+\bfx/n)\to_{n\to\infty}-\frac{\partial^k}{\partial x_{i_1},\dots,\partial x_{i_k}}\norm{\bfx}_D
$$
and the desired result obtains.
}
\end{exam}
Let $C$ be a copula and $C^n$ be the copula of the corresponding componentiwise maxima, see \eqref{eqn:componentwise_max}.
We recall that $C^{(n)}(\bfu):=C^n(\bfu^{1/n})$, $\bfu \in [0,1]^d$.
Assume that $C\in\mathcal{D}(C_G)$, where $C_G$ is an extreme-value copula.
A readily demonstrable result implied by Theorem \ref{theo:extension to copulas in a neighborhood} is the convergence of $C^{(n)}$ to $C_G$ in variational distance.
\begin{cor}\label{cor: max_copula}
Assume $C$ satisfies conditions \eqref{eqn:crucial expansion of copula} and \eqref{eq:limit_under_der}, with continuous partial derivatives of order up to $d$ on $(0,1)^d$, then
$$
\sup_{A \in \mathbb{B}^d\cap [0,1]^d}|C^{(n)}(A)-C_G(A)|\to_{n\to\infty}0.
$$
\end{cor}
\begin{proof}
For any $\bfu \in [0,1]^d$, define
$$
\widetilde{C}^{(n)}(\bfu):=P\left(n\left(\bfM^{(n)}-\bfone\right)\leq \log \bfu \right)=C^n(1+\log \bfu /n).
$$
By Theorem \ref{theo:extension to copulas in a neighborhood}, $\widetilde{C}^{(n)}$ converges to $C_G$ in variational distance.
Now, for some $\varepsilon\in(0,1)$, set
$$
\mathcal{U}_\varepsilon:=\cup_{j=1}^d\{\bfu \in [0,1]: u_j <\varepsilon \text{ or } u_j>1-\varepsilon\}.
$$
In particular, fix $\varepsilon>0$ such that $C_G(\mathcal{U}_\varepsilon^\complement)>1-\varepsilon_0$, for some arbitrarily small $\varepsilon_0\in(0,1)$.
Then, using the Taylor expansion $u^{1/n}=1+n^{-1}\log u+o(1/n)$, with uniform reminder over
$\mathcal{U}_\varepsilon^\complement$, together with the Lipschitz continuity of $C$,
we obtain
$$
\sup_{\bfu \in \mathcal{U}_\varepsilon^\complement}\left|C^{(n)}(\bfu)-\widetilde{C}^{(n)}(\bfu)\right|\to_{n \to \infty}0,
$$
and therefore
$
\lim\sup_{n \to \infty} C^{(n)}(\mathcal{U}_\varepsilon)<\varepsilon_0.
$
This implies that, as $n \to \infty$, we have
\begin{equation}\label{eqn:copula_inequalities}
\sup_{A \in \mathbb{B}^d\cap [0,1]^d}\left|C^{(n)}(A)-C_G(A) \right| \leq
\sup_{A \in \mathbb{B}^d\cap \,\mathcal{U}_\varepsilon^\complement}\left|C_\varepsilon^{(n)}(A)-\widetilde{C}_\varepsilon^{(n)}(A)\right|+O(\varepsilon_0),
\end{equation}
where $C_\varepsilon^{(n)}$ and $\widetilde{C}^{(n)}_\varepsilon$ are the normalised versions $C_\varepsilon^{(n)}=C^{(n)}/C^{(n)}(\mathcal{U}_\varepsilon^\complement)$ and
$\widetilde{C}^{(n)}_\varepsilon=\widetilde{C}^{(n)}/\widetilde{C}^{(n)}(\mathcal{U}_\varepsilon^\complement)$.
Finally, denote their densities by $c_\epsilon^{(n)}$ and $\widetilde{c}_\varepsilon^{(n)}$, respectively. Then, the supremum on the right hand side in \eqref{eqn:copula_inequalities} is attained at the set
$$
\widetilde{\mathcal{U}}_\varepsilon^{(n)}:=\left\{\bfu \in \mathcal{U}_\varepsilon^\complement: c_\varepsilon^{(n)}(\bfu)>\widetilde{c}_\varepsilon^{(n)}(\bfu)\right\}.
$$
Notice that $c_\varepsilon^{(n)}$ and $\widetilde{c}_\varepsilon^{(n)}$ are both positive on $\mathcal{U}_\varepsilon^\complement$, for $n$ sufficiently large.
Following steps similar to those in the proof of Theorem \ref{theo:extension to copulas in a neighborhood} and exploiting the continuity of the partial derivatives of $C$,  we obtain
$$
c_\varepsilon^{(n)}(\bfu)/ \widetilde{c}_\varepsilon^{(n)}(\bfu)\to_{n \to \infty}1,
$$
for all $\bfu \in \mathcal{U}_\varepsilon^\complement$.
Therefore, $\widetilde{\mathcal{U}}_\varepsilon^{(n)} \downarrow\emptyset$ as $n\to\infty$ and the result follows.
\end{proof}
%

%
%
\section{The General Case}\label{sec:general case}
%
%

Let $\bfX=(X_1,\dots,X_d)$ be a rv with arbitrary df $F$. By Sklar's theorem (\citealt{sklar59, sklar96}) we can assume the representation
\[
\bfX=(X_1,\dots,X_d)=\left(F_1^{-1}(U_1),\dots,F_d^{-1}(U_d)\right),
\]
where $F_i$ is the df of $X_i$, $i=1,\dots,d$, and $\bfU=(U_1,\dots,U_d)$ follows a copula, say $C$, corresponding to $F$.

Let $\bfX^{(1)},\bfX^{(2)},\dots$ be independent copies of $\bfX$ and let $\bfU^{(1)},\bfU^{(2)},\dots$ be independent copies of $\bfU$. Again we can assume the representation
\[
\bfX^{(i)}=\left(X_1^{(i)},\dots,X_d^{(i)}\right)=\left(F_1^{-1}\left(U_1^{(i)}\right),\dots,F_d^{-1}\left(U_d^{(i)}\right)\right),\qquad i=1,2,\dots
\]
From the fact that each quantile function $F_i^{-1}$ is monotone increasing, we obtain
\begin{align*}
&\bfM^{(n)}\\
&:= \left(\max_{1\le i\le n}X_1^{(i)},\dots,\max_{1\le i\le n}X_d^{(i)} \right)\\
&=\left(\max_{1\le i\le n}F_1^{-1}\left(U_1^{(i)}\right),\dots,\max_{1\le i\le n}F_d^{-1}\left(U_d^{(i)} \right) \right)\\
&= \left(F_1^{-1}\left(\max_{1\le i\le n}U_1^{(i)} \right),\dots, F_d^{-1}\left(\max_{1\le i\le n}U_d^{(i)} \right) \right)\\
&= \left(F_1^{-1}\left(1+\frac 1n \left(n \left(\max_{1\le i\le n}U_1^{(i)}-1\right) \right) \right),\dots, F_d^{-1}\left(1+\frac 1n \left(n \left(\max_{1\le i\le n}U_d^{(i)}-1\right) \right) \right)  \right).
\end{align*}

Theorem \ref{theo:main theorem for GPD} now implies the following result.

\begin{prop}\label{prop:multi}
Let $\bfeta=(\eta_1,\dots\eta_d)$ be a rv with standard multivariate max-stable df $G(\bfx)=\exp(-\norm{\bfx}_D)$, $\bfx\le\bfzero\in\R^d$. Let $\bfX$ be a rv with some distribution $F$ and a copula $C$.
Suppose that either $C$ is a GPC with corresponding $D$-norm $\norm\cdot_D$, which has partial derivatives of order $d\ge 2$, or $C$ satisfies conditions \eqref{eqn:crucial expansion of copula} and \eqref{eq:limit_under_der}. Then,
\begin{align*}
&\sup_{A\in\mathbb B^d}\abs{P\left(\bfM^{(n)}\in A\right)- P\left(\left(F_1^{-1}\left(1+\frac 1n\eta_1 \right),\dots, F_d^{-1}\left(1+\frac 1n \eta_d\right)  \right)\in A\right)}\\
&\to_{n\to\infty}0.
\end{align*}
\end{prop}

Finally, we generalise the result in Proposition \ref{prop:multi} to the case where the rv of componentwise maxima is suitably normalised.
Precisely, we now consider the case that $F \in \mathcal{D}(G_{\bfgamma}^*)$,
i.e. $F$ belongs to the domain of attraction of a {\it generalised} multivariate max-stable df $G_{\bfgamma}^*$, with tail index $\bfgamma=(\gamma_1,\ldots,\gamma_d)$, e.g. \citet[][Ch. 4]{falkpadoan2018}.
This means that there exist sequences of norming vectors $\bfa_n=(a_{n}^{(1)},\ldots,a_{n}^{(d)})>\bfzero$ and $\bfb_n=(b_n^{(1)},\ldots,b_n^{(d)})\in\R^d$, for $n\in\N$, such that $(\bfM^{(n)}-\bfb_n)/\bfa_n \to_{D} \bfY$ as $n \to \infty$, where $\bfY$ is a rv with distribution $G_{\bfgamma}^*$. The copula of
$G_{\bfgamma}^*$ is the extreme-value copula in
\eqref{eqn:extreme_value_copula} and its margins $G_{\gamma_j}^*$, $j=1,\ldots,d$, are members of the generalised extreme-value family of dfs in \eqref{eq:gev_family}.

To attain the convergence in variational distance, we combine Proposition \ref{prop:multi}, obtained under conditions involving only dependence structures, with univariate von Mises conditions on the margins $F_1,\ldots, F_d$, see \eqref{eq: conv_dens}-\eqref{eq: restr_gumb}. We denote by $\bfx_0:=(x_{0}^{(1)}, \ldots, x_{0}^{(d)})$, where $x_{0}^{(j)}:=\sup\{x \in \R: \, F_j(x)<1\}$,  $j=1,\ldots,d$, the vector of endpoints.
\begin{cor}\label{cor: conv_rescaled}
Let $\bfY$ and $\bfX$ be rvs with a generalised multivariate max-stable df $G_{\bfgamma}^*$ and a continuous df $F$, respectively. Assume that $F \in \mathcal{D}(G_{\bfgamma}^*)$
and that its copula $C$ satisfies the assumptions of Proposition \eqref{prop:multi}. Assume further
that, for $1\leq j\leq d$, the density of the $j$-th margin $F_{j}$ of $F$ satisfies one of the conditions \eqref{eq: restr_frec}-\eqref{eq: restr_gumb} with $f'$, $\gamma$ and $x_0$ replaced by $f_j'$, $\gamma_j$ and $x_{0}^{(j)}$. Then,
\begin{align*}
&\sup_{A\in\mathbb B^d}\abs{P\left(\frac{\bfM^{(n)}-\bfb_n}{\bfa_n}\in A\right)- P\left(\bfY\in A\right)}\to_{n\to\infty}0.
\end{align*}
\end{cor}
\begin{proof}
Let $\bfeta=(\eta_1, \ldots, \eta_d)$ be rv with standard multivariate max-stable distribution
$G(\bfx)=\exp(-\norm{\bfx}_D)$.
Define,
$$
\bfY_n:= \left(\frac{1}{a_n^{(1)}}\left(F_1^{-1}\left(1+\frac 1n \eta_1\right)-b_n^{(1)}\right), \ldots, \frac{1}{a_n^{(d)}}\left(F_d^{-1}\left(1+\frac 1n \eta_d\right)-b_n^{(d)}\right)\right).
$$
Observe that
\begin{align}
\nonumber
\sup_{A\in\mathbb B^d}\abs{P\left(\frac{\bfM^{(n)}-\bfb_n}{\bfa_n}\in A\right)- P\left(\bfY\in A\right)} \leq T_{1,n}+T_{2,n},
\end{align}
where
$$
T_{1,n}:=\sup_{A\in\mathbb B^d}\abs{P\left(\bfM^{(n)}\in A\right)- P\left(\left(F_1^{-1}\left(1+\frac 1n\eta_1 \right),\dots, F_d^{-1}\left(1+\frac 1n \eta_d\right)  \right)\in A\right)}
$$
and
$$
T_{2,n}:=\sup_{A\in\mathbb B^d}\abs{P\left(\bfY_{n}\in A\right)- P\left( \bfY \in A\right)}.
$$
By Proposition \ref{prop:multi}, $T_{1,n} \to_{n \to \infty}0$. To show that $T_{2,n} \to_{n \to \infty}0$, it is sufficient to show  pointwise convergence of the probability density function of $\bfY_n$ to that of $\bfY$ and then to appeal to the Scheff\'{e}'s lemma.
First, notice that $G_{\bfgamma}^*$ and $G$ have the same extreme-value copula. Thus, from \eqref{eqn:extreme_value_copula} it follows that, for $\bfx \in \R^d$,
$
G_{\bfgamma}^*(\bfx)=G(\bfu(\bfx)),
$
where $\bfu(\bfx)=\left(u^{(1)}(x_1), \ldots, u^{(d)}(x_d)\right)$ with $u^{(j)}(x_j)=\log G^*_{\gamma_j}(x_j)$ for $j=1,\ldots,d$.
Now, define
$
Q^{(n)}(\bfx):=P(\bfY_n\leq \bfx)=G(\bfu_n(\bfx)),
$
for $\bfx \in \mathbb{R}^d$, where $\bfu_{n}(\bfx)=\left(u_n^{(1)}(x_1), \ldots,u_n^{(d)}(x_d) \right)$ with
$$
u_n^{(j)}(x_j):=-n\left(1-F_j\left(a_n^{(j)}x_j+b_n^{(j)}\right)\right), \quad 1\leq j\leq d.
$$
Consequently, as $n \to \infty$,
\begin{align}
\nonumber
\frac{\partial^d}{\partial x_1 \dots \partial x_d} Q^{(n)}(\bfx)&=
g(\bfu_n(\bfx)) \prod_{j=1}^d \frac{n a_n^{(j)}F_j\left(a_n^{(j)}x_j+b_n^{(j)}\right)^{n-1} f_j\left(a_n^{(j)}x_j+b_n^{(j)}\right)}{F_j\left(a_n^{(j)}x_j+b_n^{(j)}\right)^{n-1}}\\
\nonumber
\nonumber& \simeq g(\bfu(\bfx)) \prod_{j=1}^d \frac{g^*_{\gamma_j}(x_j)}{G^*_{\gamma_j}(x_j)}\\
\nonumber&= \frac{\partial^d}{\partial x_1 \dots \partial x_d} G(\bfu(\bfx))=\frac{\partial^d}{\partial x_1 \dots \partial x_d}G_{\bfgamma}^*(\bfx),
\end{align}
where $g$ is as in \eqref{eq:faadibruno_g} and $g^*_{\gamma_j}(x)=(\partial/\partial x)G^*_{\gamma_j}(x)$, $1\leq j\leq d$. In particular, the second line follows from the continuity of $g$ and Proposition 2.5 in \cite{resn08}. The proof is now complete.
\end{proof}
%

%
\section{Applications}\label{sec:applications}
%
%
%
 The strong convergence results established in Sections \ref{sec:strong results for copulas} and \ref{sec:general case} can be used to refine asymptotic statistical theory for extremes.
 Max-stable distributions have been used for modelling extremes in several statistical analyses (e.g. \citealt[Ch. 8]{coles2001}; \citealt[Ch. 9]{beirgotese04}; \citealt{marcon2017}; \citealt{mhalla2017} to name a few). Parametric and nonparametric inferential procedures have been proposed for fitting max-stable models to the data (e.g., \citealt{gudendorf2012}; \citealt{berbude13}; \citealt{marcon2017}; \citealt{d2017_B}). The asymptotic theory of the corresponding estimators is well established assuming that a sample of (componentwise) maxima follows a max-stable distribution. In practice, the latter provides only an approximate distribution for sample maxima. The recent results in \citet{ferreira2015}, \citet{dombry2015}, \citet{bucher2018} and \cite{berghaus2018} account for such model misspecification, in the univariate setting.
 In the multivariate case, in \citet{Bucher2014}, weak convergence and consistency in probability of empirical copulas, under suitable second order conditions (\citealt[e.g.][]{BUCHER2019}), have been studied. This is the only multivariate contribution focusing on the problem of convergences, under model misspecification, as far as we known.
In the sequel, we illustrate how our variational convergence results, obtained under conditions \eqref{eqn:crucial expansion of copula} and \eqref{eq:limit_under_der},  allow to establish a stronger form of consistency, for both frequentist and Bayesian procedures. To do that, we resort to the notion of remote contiguity.

\begin{defi}
(\citealt{kleijn2017}) For $k\in\N$, let $r_k,s_k$ be real valued sequences such that
$0 < r_k,s_k \to_{k \to \infty} 0$. Let $\mu_k$ and $\nu_k$ be sequences of probability measures. Then, $\nu_k$ is said $r_k$-to-$s_k$-remotely contiguous with respect to $\mu_k$ if $\mu_k(E_k) = o(r_k)$, for a sequence of measurable events $E_k$, implies $\nu_k(E_k)=o(s_k)$. In this case, we write $s_k^{-1}\nu_k \lhd r_k^{-1}\mu_k$.
\end{defi}

\subsection{Frequentist approach}\label{sec:freq_app}
%
Let $\Theta$ denote a parameter space (possibly infinite dimensional) and
$\theta \in \Theta$ be a parameter of interest.
Let $\bfY$ be a $d$-dimensional rv with a df $F$, pertaining to a probability measure $\mu_0$ on $\mathbb{B}^d$. Denote by $\mu_k$ the corresponding $k$-fold product measure.
Let $\bfY^{(1:k)}=(\bfY^{(1,k)}, \ldots,\bfY^{(k,k)})$ be
a sequence of $k$ iid copies of $\bfY$. Consider a measurable map $T_k: \times_{i=1}^k\mathbb{R}^d \to \Theta$ and let
$$
\widehat{\theta}_k:=T_k(\bfY^{(1:k)})
$$
be an estimator of $\theta$.
%
%
%
Let $\mathscr{D}$ denote a metric on $\Theta$.

If for every $\varepsilon>0$ there are constants $c_\varepsilon,c_\varepsilon'>0$ such that $
\mu_k(\mathscr{D}(\widehat{\theta}_k,\theta)>\varepsilon)=o(e^{-c_\varepsilon k})
$
and $k^{1+c_\varepsilon'}\nu_k \lhd e^{c_\varepsilon k} \mu_k$, then, we can conclude by
Borel-Cantelli lemma that
$$
\mathscr{D}(T_k(\bfZ^{(1:k)}), \theta)\to_{k \to \infty}0,\qquad \nu_k\text{-almost surely},
$$
where $\bfZ^{(1:k)}=(\bfZ^{(1,k)}, \ldots,\bfZ^{(k,k)})$ is a sequence of iid rv with common probability measure $\nu_{0,k}$ on $\mathbb{B}^d$, and $\nu_k$ is the corresponding $k$-fold product measure.
 The required form of remote contiguity easily obtains if
 $\sup_{A \in \mathbb{B}^d}|\nu_{0,k}(A)-\mu_0(A)|\to_{k \to \infty}0$, $\mu_0$ and $\nu_{0,k}$ have the same support and continuous Lebesgue densities, $p_{0,k}$ and $m_0$, satisfying
 \begin{equation}\label{eq: integrability}
 \sup_{k \geq k_0} \rho_\delta(\nu_{0,k},\mu_0):=\sup_{k \geq k_0}\int_{\mathcal{X}_{\delta,k}}\left(
 p_{0,k}(\bfx)/ m_{0}(\bfx)
 \right)^\delta p_{0,k}(\bfx)\diff\bfx <\infty,
 \end{equation}
 for some $\delta\in (0,1]$ and $k_0\in \mathbb{N}$, where $\mathcal{X}_{\delta,k}=\{\bfx \in \mathbb{R}^d:  p_{0,k}(\bfx)/ \mu_{0}(\bfx)>e^{1/\delta}\}$.
 Essentially, variational convergence and \eqref{eq: integrability} guarantee that the fourth moments and the expectations of the triangular array of variables $ \{\log p_{0,k}(\bfZ^{(i,k)})-\log m_0(\bfZ^{(i,k)}), 1 \leq i \leq  k; k\geq k_0+k_0'\}$ are uniformly bounded and asymptotically null, respectively, for a sufficiently large $k_0' \in \mathbb{N}$. The corresponding sequence of (rescaled) log-likelihood ratios, then, converges to 0 by the strong law of large numbers.

This novel asymptotic technique can be fruitfully applied to parameter estimation problems for multivariate max-stable models. In this context, the probability measure $\mu_0$ can be associated to a multivariate max-stable df $G_{\bfgamma}^*$ or to its extreme-value copula. Accordingly, we see the probability measure $\nu_{0,k}$ as associated to the df of a normalized rv of componentwise maxima,
computed over a number of underlying rv indexed by $k$, say $n_k$.

Exploiting Corollary \ref{cor: max_copula}, herein we specialise the above procedure to the estimation of an extreme-value copula via the empirical copula of sample maxima.

First, we recall some basic notions. Let
$\bfZ^{(1:k)}$ be a sequence of iid copies of a rv $\bfZ$ with some copula $C$.
Then, the empirical copula function $\widehat{C}_k$  is a map $T_k:\times_{i=1}^k\mathbb{R}^d \mapsto \ell^\infty([0,1]^d)$
defined by
\begin{equation*}\label{eq:empirical_copula}
\begin{split}
&\widehat{C}_k(\bfu; \bfZ^{(1:k)}):=(T_k(\bfZ^{(1:k)}))(\bfu) \\
&=\frac{1}{k}\sum_{i=1}^k \indic\left(\frac{\sum_{l=1}^k {\indic(Z^{(l,k)}_1 \leq Z^{(i,k)}_1)}}{k}
 \leq \,u_1, \ldots, \frac{\sum_{l=1}^k \indic(Z^{(l,k)}_d \leq Z^{(i,k)}_d)}{k} \leq \,u_d\right),
\end{split}
\end{equation*}
for $\bfu \in [0,1]^d$, with $\indic(E)$ denoting the indicator function of the event $E$.
\begin{prop}\label{prop: consistency}
Let $\bfM^{(n)}=(M_1^{(n)},\ldots,M_d^{(n)})$, $C$ and $G$ be as in Proposition \ref{prop:multi}, with $C$ satisfying the assumptions of Corollary \ref{cor: max_copula}.
Let $\bfM^{(n,1:k)}=(\bfM^{(n,1)}, \ldots, \bfM^{(n,k)})$ be $k$ independent copies of $\bfM^{(n)}$, with $n\equiv n_k \to_{k \to \infty} \infty$. Assume that $C^{(n)}$ and $C_G$
satisfy
\begin{equation}\label{eq:cond_int}
\sup_{k\geq k_0}\rho_\delta(C^{(n)},C_G)<	\infty,
\end{equation}
for some $\delta \in (0,1]$, $k_0 \in \mathbb{N}$, with $\rho_\delta$ as in \eqref{eq: integrability}.
%
Then, almost surely
$$
\sup_{\bfu \in [0,1]^d}\left|\widehat{C}_k(\bfu)- C_G(\bfu)\right|\to_{k \to \infty}0,
$$
where $\widehat{C}_k \equiv \widehat{C}_k(\cdot; \bfM^{(n,1:k)})$.
\end{prop}
For the proof of Proposition of \ref{prop: consistency} we establish  the following remote contiguity relation.
\begin{lemma}\label{lem:remote}
	Let $C^{(n,k)}$ and $C^k_G$ denote the $k$-fold product measures pertaining to $C^{(n)}$ and $C_G$, respectively. Then, $k^{2}C^{(n,k)} \lhd e^{ck} C^k_G$, for any $c>0$.
\end{lemma}
\begin{proof}
	Let $E_k$, $k=1,2, \ldots$ be a sequence of measurable events satisfying $C_G^{k}(E_k)=o(e^{-ck})$, for some $c>0$. It is not difficult to see that, for any $\varepsilon>0$,
	$$
	C^{(n,k)}(E_k) \leq e^{\varepsilon k} C_G^k(E_k)+C^{(n,k)}( S_k>\varepsilon k ),
	$$
	where $S_k = \sum_{i=1}^k \log \left\{c^{(n)}(\bfU^{(n,i)})/c_G(\bfU^{(n,i)})\right\}$, $\bfU^{(n,i)}, \, 1 \leq i \leq k,$ are iid according to $C^{(n)}$, $c^{(n)}$ and $c_G$ are the Lebesgue densities of $C^{(n)}$ and $C_G$, respectively. Choosing $\varepsilon<c$, the first term on the right-hand side is of order $o(e^{-(c-\varepsilon)k})$. As for the second term,  as $k \to +\infty$ we have that $n\equiv n_k\to\infty$ and, by Corollary \ref{cor: max_copula}, $
	\varepsilon_k:=\sup_{A \in \mathbb{B}^d\cap [0,1]^d}|C^{(n)}(A)-C_G(A)|=o(1)$. Thus, defining
	$$
	\eta_{\alpha,k}:=E\left[
	\log^\alpha \left\{\frac{c^{(n)}(\bfU^{(n,1)})}{c_G(\bfU^{(n,1)})}\right\}
	\right], \quad \alpha \in \mathbb{N},
	$$
	under assumption \eqref{eq:cond_int}, Theorem 6 in \cite{wong1995} guarantees that, as $k \to +\infty$ ,
	$
	\max(\eta_{1,k}, \eta_{2,k})= O(\varepsilon_k \log^2(1/\varepsilon_k))  \leq \varepsilon/2.
	$
	Furthermore, simple analytical derivations lead to show that
	$$
	\sup_{k \geq k_0} (-\eta_{3,k})\leq 1+ \sup_{k \geq k_0}\eta_{4,k} \leq 2 + \log^4(K)+\sup_{k \geq k_0}\rho_\delta(C^{(n)},C_G)<+\infty,
	$$
	for some large but fixed $K>e^{1/\delta}$.
	Together with triangular and Markov inequalities, these facts entail that as $k \to +\infty$
	\begin{equation*}
	\begin{split}
	C^{(n,k)}( S_k>\varepsilon k ) & \leq C^{(n,k)}( |S_k-k\eta_{1,k}|>\varepsilon/2 k )\\
	& \leq \left(\frac{2}{\varepsilon k}\right)^4E\left[ (S_k-k\eta_{1,k})^4\right]\\
	& \leq \left(\frac{2}{\varepsilon }\right)^4
		\left[\frac{1}{k^3}(\eta_{4,k}-4\eta_{1,k}\eta_{3,k}+6\eta_{1,k}^2 \eta_{2,k})+\frac{3}{k^2}(\eta_{2,k}-\eta_{1,k})^2
		\right]\\
		&=o(k^{-2}),
	\end{split}
	\end{equation*}
where, in the third line, we exploit nonnegativity of $\eta_{1,k}$. The result now follows.
\end{proof}
\begin{proof}[Proof of  Proposition \ref{prop: consistency}]
Let $\bfV$ be a rv distributed according to the extreme-value copula $C_G$. Let
$\bfV^{(1:k)}=(\bfV^{(1)}, \ldots, \bfV^{(k)})$ be a sequence of iid copies of $\bfV$
with joint distribution $C_G^{(k)}$.
Then, standard empirical process arguments (\citealt{gudendorf2012}, \citealt{10.1007/BFb0097426}, \citealt{wellner92}) yield that, for any $\varepsilon>0$,
\begin{equation*}
\begin{split}
&C_G^{(k)}\left(\sup_{\bfu \in [0,1]^d}\left|\widehat{C}_k(\bfu; \bfV^{(1:k)})- C_G(\bfu)\right|>\varepsilon\right) \\
&\quad \leq 2d\exp\left(-\frac{b_\varepsilon^2k}{(d+1)^2}\right)+16\frac{kb_\varepsilon^2}{(d+1)^2} \exp\left(-\frac{2b_\varepsilon^2k}{(d+1)^2}\right)
\end{split}
\end{equation*}
for some $b_\varepsilon\in(0, \varepsilon)$.
The term on the right hand side is of order $O(e^{-c_\varepsilon k})$, for some $c_\varepsilon>0$.
By Lemma \ref{lem:remote}, we have that $k^2 C^{(n,k)}\lhd e^{ck}C_G^{(k)}$ for all $c>0$, where $C^{(n,k)}$ is the $k$-fold product measure corresponding to $C^{(n)}$.
Let $\bfU^{(n,1:k)}=(\bfU^{(n,1)}, \ldots, \bfU^{(n,k)})$, where
$$
\bfU^{(n,i)}=\left(F_1\left(M_1^{(n,i)}\right)^n, \ldots,F_d\left(M_d^{(n,i)}\right)^n\right), \quad i=1, \ldots,k.
$$
The result now follows observing that $\bfU^{(n,1:k)}$ is distributed according to $C^{(n,k)}$ and that
$
\widehat{C}_k(\bfu)\equiv \widehat{C}_k(\bfu; \bfM^{(n, 1:k)})=\widehat{C}_k(\bfu; \bfU^{(n,1:k)}).
$
\end{proof}
\begin{rema}
{\upshape
Notice that the assumption in \eqref{eq:cond_int} of Proposition \ref{prop: consistency} is not overambitious. Indeed, when $C^{(n)}$ is obtained from copulas that are extreme-value copulas, the required condition is always satisfied. While, when $C^{(n)}$ is obtained from copulas that are in the domain of attraction of extreme-value copulas, analytically verifying \eqref{eq:cond_int} seems troublesome. Still, numerically checking whether some copula models meet this asumption can be fairly simple.
For instance, consider the copula of Example \ref{exam:DoA_copula}, given in equation \eqref{eq:exponential_copula}, and let $c$ denote its density. Denote by
$c^{(n)}$ the density of the copula $C^{(n)}$ pertaining to $C$
%
and by $c_G$ the density of the extreme-value copula model in  \eqref{eq:limit_exp_copula}. In this case, Corollary \ref{cor: max_copula} applies and $C^{(n)}$ converges to $C_G$ in variational distance. Figure \ref{fig:copula_densities} displays the plots of the densities $c$, $c_G$ and $c^{(n)}$, with $n=100$. Outside a neighborhood of the origin, pointwise convergence of $c^{(n)}$ to $c_G$ turns out to be quite fast. In addition, the middle-right to bottom-right panels show that the density ratio $c^{(n)}/c_G$ is uniformly bounded by a finite constant, as the sample size $n$ increases. Consequently, the condition in \eqref{eq:cond_int} is satisfied.
}
\end{rema}
%
%
%
\begin{figure}[h!]
	\centering
	\includegraphics[width=0.40\textwidth, page=1]{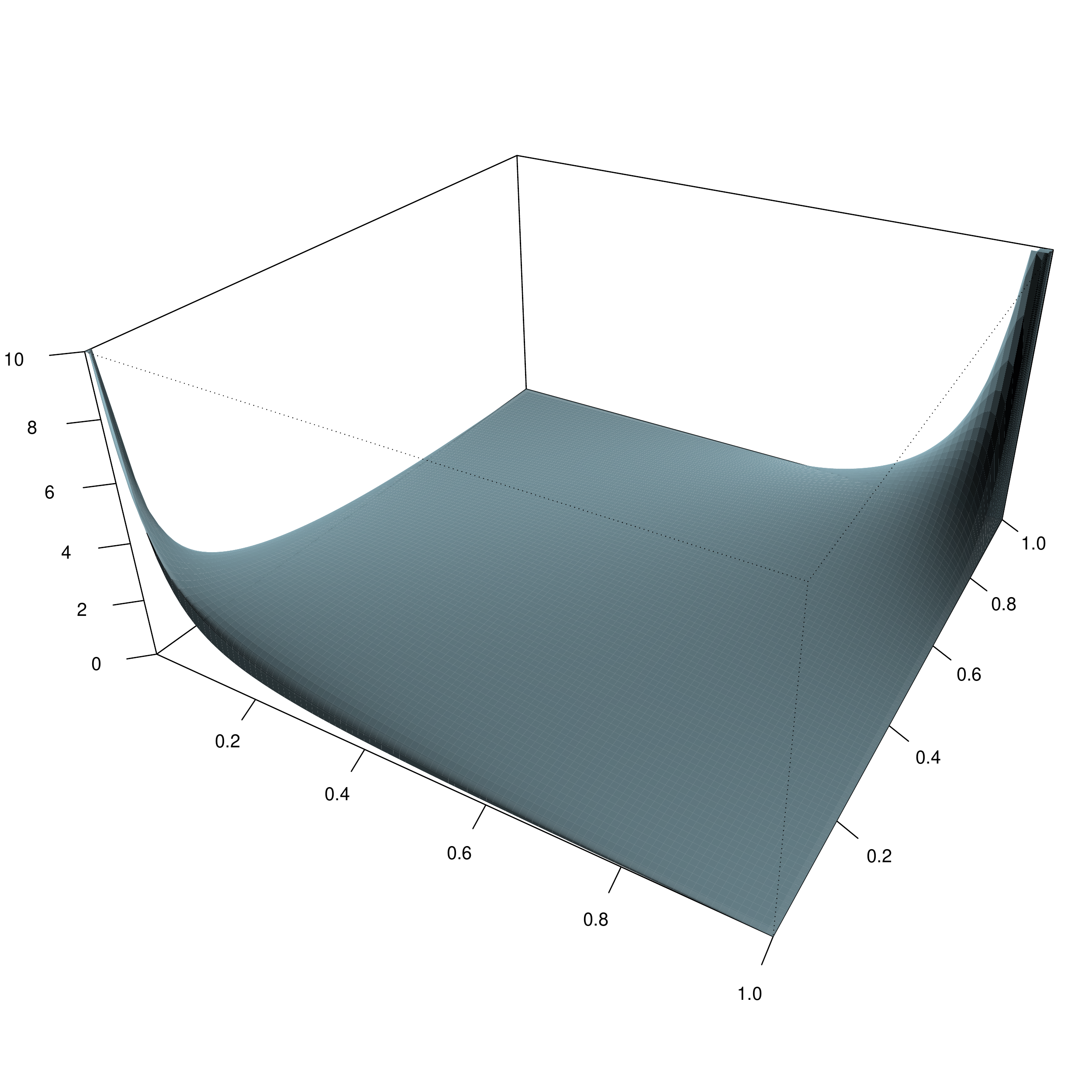}
	\includegraphics[width=0.40\textwidth, page=2]{plots3.pdf}\\
	\includegraphics[width=0.40\textwidth, page=3]{plots3.pdf}
	\includegraphics[width=0.42\textwidth, page=4]{plots3.pdf}\\
    \includegraphics[width=0.40\textwidth, page=5]{plots3.pdf}
    \includegraphics[width=0.40\textwidth, page=6]{plots3.pdf}
	\caption{Top-left and -right panels display the densities $c_G$ and $c$ of the copula models in \eqref{eq:limit_exp_copula} and in \eqref{eq:exponential_copula}, respectively. Middle-left panel shows the density $c^{(n)}$ of the copula $C^{(n)}$ pertaining to the copula model in \eqref{eq:exponential_copula}, with sample size $n=100$. Middle-right to bottom right panels depict the density ratio $c^{(n)}/c_G$, for $n=2,50,100$, respectively.}
	\label{fig:copula_densities}
\end{figure}
%
%

%
\subsection{Bayesian approach}\label{sec:Bayes_app}
A similar scheme is exploited by \citet{padoanrizzelli2019} in a Bayesian context, where extended Schwartz' theorem, e.g. \citet[][Theorem 6.23]{gvdv_2017}, provides with exponential bounds for posterior concentration in a neighborhood of the true parameter.
In particular, \citet{padoanrizzelli2019} consider a nonparametric Bayesian approach for estimating the $D$-norm $\norm\cdot_D$ and the densities of the associated angular measure, see \citet[][pp. 25--29]{falk2019}. Therein, Corollary \ref{cor: conv_rescaled} is leveraged to obtain a suitable remote contiguity result, allowing to extend
almost-sure consistency of the proposed estimators from the case of data following a max-stable model, to the case of suitably normalised sample maxima, whose distribution lies in a variational neighbourhood of the latter.

\section*{Acknowledgements}

The authors are indebted to the Associate Editor and two anonymous reviewers for their careful reading of the manuscript and their constructive remarks.
Simone A. Padoan is supported by the Bocconi Institute for Data Science and Analytics (BIDSA).

\bibliographystyle{chicago}
\bibliography{evt_v2}

\begin{thebibliography}{}

\bibitem[\protect\citeauthoryear{Beirlant, Goegebeur, Segers, and
  Teugels}{Beirlant et~al.}{2004}]{beirgotese04}
Beirlant, J., Y.~Goegebeur, J.~Segers, and J.~Teugels (2004).
\newblock {\em Statistics of Extremes: Theory and Applications}.
\newblock Wiley Series in Probability and Statistics. Chichester, UK: Wiley.

\bibitem[\protect\citeauthoryear{Berghaus and B{\"u}cher}{Berghaus and
  B{\"u}cher}{2018}]{berghaus2018}
Berghaus, B. and A.~B{\"u}cher (2018).
\newblock Weak convergence of a pseudo maximum likelihood estimator for the
  extremal index.
\newblock {\em Ann. Statist.\/}~{\em 46\/}(5), 2307--2335.

\bibitem[\protect\citeauthoryear{Berghaus, B{\"u}cher, and Dette}{Berghaus
  et~al.}{2013}]{berbude13}
Berghaus, B., A.~B{\"u}cher, and H.~Dette (2013).
\newblock Minimum distance estimators of the \uppercase{P}ickands dependence
  function and related tests of multivariate extreme-value dependence.
\newblock {\em J. SFdS\/}~{\em 154\/}(1), 116--137.

\bibitem[\protect\citeauthoryear{B{\"u}cher and Segers}{B{\"u}cher and
  Segers}{2014}]{Bucher2014}
B{\"u}cher, A. and J.~Segers (2014).
\newblock Extreme value copula estimation based on block maxima of a
  multivariate stationary time series.
\newblock {\em Extremes\/}~{\em 17\/}(3), 495--528.

\bibitem[\protect\citeauthoryear{B\"{u}cher and Segers}{B\"{u}cher and
  Segers}{2018}]{bucher2018}
B\"{u}cher, A. and J.~Segers (2018).
\newblock {Maximum likelihood estimation for the Fr\'{e}chet distribution based
  on block maxima extracted from a time series}.
\newblock {\em Bernoulli\/}~{\em 24\/}(2), 1427--1462.

\bibitem[\protect\citeauthoryear{B{\"u}cher, Volgushev, and Zou}{B{\"u}cher
  et~al.}{2019}]{BUCHER2019}
B{\"u}cher, A., S.~Volgushev, and N.~Zou (2019).
\newblock On second order conditions in the multivariate block maxima and peak
  over threshold method.
\newblock {\em Journal of Multivariate Analysis\/}~{\em 173}, 604 -- 619.

\bibitem[\protect\citeauthoryear{Coles}{Coles}{2001}]{coles2001}
Coles, S. (2001).
\newblock {\em {An Introduction to Statistical Modeling of Extreme Values}},
  Volume 208.
\newblock Springer.

\bibitem[\protect\citeauthoryear{de~Haan and Peng}{de~Haan and
  Peng}{1997}]{DEHAAN1997195}
de~Haan, L. and L.~Peng (1997).
\newblock Rates of convergence for bivariate extremes.
\newblock {\em Journal of Multivariate Analysis\/}~{\em 61\/}(2), 195 -- 230.

\bibitem[\protect\citeauthoryear{Deheuvels}{Deheuvels}{1980}]{10.1007/BFb0097426}
Deheuvels, P. (1980).
\newblock Non parametric tests of independence.
\newblock In J.-P. Raoult (Ed.), {\em Statistique non Param{\'e}trique
  Asymptotique}, Berlin, Heidelberg, pp.\  95--107. Springer Berlin Heidelberg.

\bibitem[\protect\citeauthoryear{Dombry}{Dombry}{2015}]{dombry2015}
Dombry, C. (2015).
\newblock Existence and consistency of the maximum likelihood estimators for
  the extreme value index within the block maxima framework.
\newblock {\em Bernoulli\/}~{\em 21\/}(1), 420--436.

\bibitem[\protect\citeauthoryear{{Dombry}, {Engelke}, and {Oesting}}{{Dombry}
  et~al.}{2017}]{d2017_B}
{Dombry}, C., S.~{Engelke}, and M.~{Oesting} (2017).
\newblock {Bayesian inference for multivariate extreme value distributions}.
\newblock {\em Electronic Journal of Statistics\/}~{\em 11}, 4813–--4844.

\bibitem[\protect\citeauthoryear{Falk}{Falk}{2019}]{falk2019}
Falk, M. (2019).
\newblock {\em Multivariate Extreme Value Theory and D-Norms}.
\newblock New York: Springer International Publishing.

\bibitem[\protect\citeauthoryear{Falk, H{\"u}sler, and Reiss}{Falk
  et~al.}{2011}]{fahure10}
Falk, M., J.~H{\"u}sler, and R.-D. Reiss (2011).
\newblock {\em Laws of Small Numbers: Extremes and Rare Events\/} (3 ed.).
\newblock Basel: Birkh{\"a}user.

\bibitem[\protect\citeauthoryear{Falk, Padoan, and Wisheckel}{Falk
  et~al.}{2019}]{falkpadoan2018}
Falk, M., S.~A. Padoan, and F.~Wisheckel (2019).
\newblock Generalized \uppercase{P}areto copulas: a key to multivariate
  extremes.
\newblock {\em Journal of Multivariate Analysis\/}~{\em 174}, 104538.

\bibitem[\protect\citeauthoryear{Ferreira and de~Haan}{Ferreira and
  de~Haan}{2015}]{ferreira2015}
Ferreira, A. and L.~de~Haan (2015, 02).
\newblock On the block maxima method in extreme value theory: Pwm estimators.
\newblock {\em Ann. Statist.\/}~{\em 43\/}(1), 276--298.

\bibitem[\protect\citeauthoryear{Ghosal and van~der Vaart}{Ghosal and van~der
  Vaart}{2017}]{gvdv_2017}
Ghosal, S. and A.~van~der Vaart (2017).
\newblock {\em Fundamentals of Nonparametric Bayesian Inference}.
\newblock Cambridge Series in Statistical and Probabilistic Mathematics.
  Cambridge University Press.

\bibitem[\protect\citeauthoryear{Gudendorf and Segers}{Gudendorf and
  Segers}{2012}]{gudendorf2012}
Gudendorf, G. and J.~Segers (2012).
\newblock Nonparametric estimation of multivariate extreme-value copulas.
\newblock {\em Journal of Statistical Planning and Inference\/}~{\em
  142\/}(12), 3073--3085.

\bibitem[\protect\citeauthoryear{Kaufmann and Reiss}{Kaufmann and
  Reiss}{1993}]{kaufr93}
Kaufmann, E. and R.-D. Reiss (1993).
\newblock Strong convergence of multivariate point processes of exceedances.
\newblock {\em Ann. Inst. Statist. Math.\/}~{\em 45\/}(3), 433--444.

\bibitem[\protect\citeauthoryear{{Kleijn}}{{Kleijn}}{2017}]{kleijn2017}
{Kleijn}, B. (2017).
\newblock {On the frequentist validity of Bayesian limits}.
\newblock {\em arXiv e-prints\/}, arXiv:1611.08444v3.

\bibitem[\protect\citeauthoryear{Marcon, Padoan, Naveau, Muliere, and
  Segers}{Marcon et~al.}{2017}]{marcon2017}
Marcon, G., S.~A. Padoan, P.~Naveau, P.~Muliere, and J.~Segers (2017).
\newblock {Multivariate nonparametric estimation of the Pickands dependence
  function using Bernstein polynomials}.
\newblock {\em Journal of Statistical Planning and Inference\/}~{\em 183},
  1--17.

\bibitem[\protect\citeauthoryear{McNeil and Ne\v{s}lehov{\'a}}{McNeil and
  Ne\v{s}lehov{\'a}}{2009}]{mcnn09}
McNeil, A.~J. and J.~Ne\v{s}lehov{\'a} (2009).
\newblock {Multivariate Archimedean copulas, $d$-monotone functions and
  $\ell_1$-norm symmetric distributions}.
\newblock {\em Ann. Statist.\/}~{\em 37\/}(5B), 3059--3097.

\bibitem[\protect\citeauthoryear{Mhalla, Chavez-Demoulin, and Naveau}{Mhalla
  et~al.}{2017}]{mhalla2017}
Mhalla, L., V.~Chavez-Demoulin, and P.~Naveau (2017).
\newblock Non-linear models for extremal dependence.
\newblock {\em Journal of Multivariate Analysis\/}~{\em 159}, 49--66.

\bibitem[\protect\citeauthoryear{Padoan and Rizzelli}{Padoan and
  Rizzelli}{2019}]{padoanrizzelli2019}
Padoan, S.~A. and S.~Rizzelli (2019).
\newblock {Strong consistency of nonparametric Bayesian inferential methods for
  multivariate max-stable distributions}.
\newblock {\em arXiv e-prints\/}, arXiv:1904.00245v2.

\bibitem[\protect\citeauthoryear{Reiss}{Reiss}{1989}]{reiss89}
Reiss, R.-D. (1989).
\newblock {\em Approximate Distributions of Order Statistics: With Applications
  to Nonparametric Statistics}.
\newblock Springer Series in Statistics. New York: Springer.

\bibitem[\protect\citeauthoryear{Resnick}{Resnick}{2008}]{resn08}
Resnick, S.~I. (2008).
\newblock {\em Extreme Values, Regular Variation, and Point Processes}.
\newblock Springer Series in Operations Research and Financial Engineering. New
  York: Springer.

\bibitem[\protect\citeauthoryear{Sklar}{Sklar}{1959}]{sklar59}
Sklar, A. (1959).
\newblock Fonctions de r{\'e}partition {\`a} $n$ dimensions et leurs marges.
\newblock {\em Pub. Inst. Stat. Univ. Paris\/}~{\em 8}, 229--231.

\bibitem[\protect\citeauthoryear{Sklar}{Sklar}{1996}]{sklar96}
Sklar, A. (1996).
\newblock Random variables, distribution functions, and copulas -- a personal
  look backward and forward.
\newblock In L.~R{\"u}schendorf, B.~Schweizer, and M.~D. Taylor (Eds.), {\em
  Distributions with fixed marginals and related topics}, Volume~28 of {\em
  Lecture Notes -- Monograph Series}, Hayward, CA, pp.\  1--14. Institute of
  Mathematical Statistics.

\bibitem[\protect\citeauthoryear{Wellner}{Wellner}{1992}]{wellner92}
Wellner, J.~A. (1992).
\newblock Empirical processes in action: A review.
\newblock {\em International Statistical Review / Revue Internationale de
  Statistique\/}~{\em 60\/}(3), 247--269.

\bibitem[\protect\citeauthoryear{Wong and Shen}{Wong and Shen}{1995}]{wong1995}
Wong, W.~H. and X.~Shen (1995).
\newblock Probability inequalities for likelihood ratios and convergence rates
  of sieve mles.
\newblock {\em Ann. Statist.\/}~{\em 23\/}(2), 339--362.

\end{thebibliography}

\end{document}